\documentclass[10pt,a4paper]{amsart}
\usepackage{graphicx}
\usepackage{amssymb}
\usepackage{amsmath}
\usepackage{amsthm}
\usepackage{amscd}
\usepackage[all,2cell]{xy}
\usepackage[markup=underlined]{changes}
\definechangesauthor[name=zy, color=cyan]{zy}

\UseRawInputEncoding
\UseAllTwocells \SilentMatrices
\newtheorem{thm}{Theorem}[section]

\newtheorem{cor}[thm]{Corollary}
\newtheorem{lem}[thm]{Lemma}
\newtheorem{exm}[thm]{Example}

\newtheorem{prop}[thm]{Proposition}
\theoremstyle{definition}
\newtheorem{defn}[thm]{Definition}
\theoremstyle{remark}
\newtheorem{rem}[thm]{\bf Remark}
\numberwithin{equation}{section}

\begin{document}
\title[The extensions of  t-structures]{The extensions of t-structures}
\author[Xiao-Wu Chen, Zengqiang Lin, Yu Zhou] {Xiao-Wu Chen, Zengqiang Lin$^*$, Yu Zhou}

\date{\today}
\thanks{$^*$The corresponding author}

\maketitle
\date{}%
\dedicatory{}%
\commby{}%
\begin{center}
\end{center}

\begin{abstract}
We  reformulate a result of Bernhard Keller on extensions of $t$-structures and give a detailed proof. In the study of hereditary $t$-structures, the notions of regular $t$-structures and global dimensions arise naturally.
\end{abstract}

\section{Introduction}
Let $\mathcal{T}$ be a triangulated category with a fixed $t$-structure $\mathcal{U}$; see \cite{BBD}. Denote by $\mathcal{T}^b$ the smallest triangulated full subcategory  containing the heart of $\mathcal{U}$. Then $\mathcal{U}$ restricts to a bounded $t$-structure  on $\mathcal{T}^b$, which is denoted by $\mathcal{U}^b$.

The following  result of extending $t$-structures is stated in \cite[Subsection~6.1]{Kel05} without a proof: any t-structure $\mathcal{V}'$ on $\mathcal{T}^b$ satisfying $d(\mathcal{U}^b, \mathcal{V}')<+\infty$ extends canonically to a $t$-structure $\mathcal{V}$ on $\mathcal{T}$. Here, we denote by $d(\mathcal{U}^b, \mathcal{V}')$ the distance \cite{QW} between two $t$-structures $\mathcal{U}^b$ and $\mathcal{V}'$ on $\mathcal{T}^b$.

The goal of this paper is to reformulate the above result as a bijective correspondence, and give a detailed proof.
\vskip 5pt

\noindent {\bf Theorem} (Keller).\; \emph{Keep the notation as above. Then there is a bijective correspondence
$$\left\{
\begin{aligned}
& t\mbox{-structures } \mathcal{V} \mbox{ on }  \mathcal{T} \\
& \mbox{with } d(\mathcal{U}, \mathcal{V})<+\infty
 \end{aligned} \right\} \longrightarrow \left
 \{
 \begin{aligned}
 & t\mbox{-structures } \mathcal{V}'   \mbox{ on }   \mathcal{T}^b
  \\
  &\mbox{with } d(\mathcal{U}^b, \mathcal{V}')<+\infty
  \end{aligned}
  \right \},
 $$
 sending $\mathcal{V}$ to its restriction $\mathcal{V}^b$ on $\mathcal{T}^b$.}

\vskip 5pt

Following \cite{Kel05}, we might view $\mathcal{V}$ as a \emph{canonical extension} of the  $t$-structure $\mathcal{V}^b$ on $\mathcal{T}^b$. Since $d(\mathcal{U}^b, \mathcal{V}^b)<+\infty$, the restricted $t$-structure $\mathcal{V}^b$ on $\mathcal{T}^b$ is necessarily bounded. We mention that the condition $ d(\mathcal{U}, \mathcal{V})<+\infty$ is needed to guarantee that  the restriction $\mathcal{V}^b$ is well defined; see Proposition~\ref{prop:dist-finite}.

The above correspondence plays a role in the proof of the triangulated structure of the orbit category, including the cluster category for an acyclic quiver; see  \cite{Kel05}. The main concern of \cite{Kel05} is hereditary $t$-structures. However, in the above correspondence, the hereditariness of $\mathcal{V}^b$ does not imply  the one of $\mathcal{V}$ in general. We observe that if the fixed $t$-structure $\mathcal{U}$ is \emph{regular}, then the above correspondence restricts a bijective correspondence between hereditary $t$-structures; see Proposition~\ref{prop:t-heredi}. We refer to Definition~\ref{defn:adm} for regular $t$-structures.

We mention the work \cite{MZ}, which studies extensions of $t$-structures on the bounded derived category of finitely generated modules to the ones on the unbounded derived category of arbitrary modules.

The paper is structured as follows. In Section~2, we recall basic facts on $t$-structures, and introduce regular $t$-structures. We study hereditary $t$-structures and the global dimension of a $t$-structure in Section~3. We study the distance between $t$-structures in Section~4. In Proposition~\ref{prop:dist-finite}, we prove that two $t$-structures with a finite distance share many properties. In the final section, we prove the theorem above.

\section{Preliminaries on t-structures}

Let $\mathcal{T}$ be a triangulated category. We denote by $\Sigma$ its translation functor and by $\Sigma^{-1}$ its quasi-inverse. Consequently, $\Sigma^n$ is defined for each integer $n$.

For a full subcategory $\mathcal{X}$ of $\mathcal{T}$, the right orthogonal subcategory is defined as $ \mathcal{X}^\perp=\{Z\in \mathcal{T}\; |\; {\rm Hom}_\mathcal{T}(X, Z)=0 \mbox{ for any } X\in \mathcal{X}\}$. The left  orthogonal subcategory $^\perp \mathcal{X}$ is defined analogously. For another full subcategory $\mathcal{Y}$ of $\mathcal{T}$, we set
$$\mathcal{X}\ast \mathcal{Y}=\{Z\in \mathcal{T}\; |\; \exists \mbox{ an exact triangle } X\rightarrow Z\rightarrow Y \rightarrow \Sigma(X) \mbox{ with } X\in \mathcal{X}, \mathcal{Y}\in \mathcal{Y}\}.$$
The operation $\ast$ is associative; see \cite[Lemma~1.3.10]{BBD}. By ${\rm Hom}_\mathcal{T}(\mathcal{X}, \mathcal{Y})=0$, we mean that ${\rm Hom}_\mathcal{T}(X, Y)=0$ for any $X\in \mathcal{X}$ and $Y\in \mathcal{Y}$.

Recall from \cite[Definition~1.3.1]{BBD} that a \emph{t-structure} $\mathcal{U}=(\mathcal{U}_{\leq 0}, \mathcal{U}_{\geq 0})$ of $\mathcal{T}$ consists of two full subcategories, which are subject to the following conditions:
\begin{enumerate}
\item[(T1)] $\Sigma\mathcal{U}_{\leq 0}\subseteq \mathcal{U}_{\leq 0}$ and $\Sigma^{-1}\mathcal{U}_{\geq 0}\subseteq \mathcal{U}_{\geq 0}$;
\item[(T2)] ${\rm Hom}_\mathcal{T}(\mathcal{U}_{\leq 0}, \Sigma^{-1}\mathcal{U}_{\geq 0})=0$;
\item[(T3)] $\mathcal{T}=\mathcal{U}_{\leq 0}* (\Sigma^{-1}\mathcal{U}_{\geq 0})$, that is, any object $Z\in \mathcal{T}$ fits into an exact triangle $$X\rightarrow Z\rightarrow \Sigma^{-1}(Y)\rightarrow \Sigma(X)$$ for some $X\in \mathcal{U}_{\leq 0}$ and $Y\in \mathcal{U}_{\geq 0}$.
\end{enumerate}
The \emph{heart} of the $t$-structure is defined to be $\mathcal{U}_0=\mathcal{U}_{\leq 0}\cap \mathcal{U}_{\geq 0}$; it is an abelian category; see \cite[Theorem~1.3.6]{BBD}.

For each integer $n$, we set $\mathcal{U}_{\leq n}=\Sigma^{-n}\mathcal{U}_{\leq 0}$ and $\mathcal{U}_{\geq n}=\Sigma^{-n}\mathcal{U}_{\geq 0}$. We observe that $(\mathcal{U}_{\leq n}, \mathcal{U}_{\geq n})$ is a $t$-structure for any integer $n$, which will be denoted by $\Sigma^{-n}(\mathcal{U})$. We have
\begin{align}\label{equ:ortho}
\mathcal{U}_{\leq n}={^\perp(\mathcal{U}_{\geq n+1})} \mbox{ and } \mathcal{U}_{\geq n+1}=(\mathcal{U}_{\leq n})^\perp.
\end{align}
Denote by $\tau_{\leq n}\colon \mathcal{T}\rightarrow \mathcal{U}_{\leq n}$ the right adjoint of the inclusion $\mathcal{U}_{\leq n}\hookrightarrow \mathcal{T}$, and by $\tau_{\geq n}\colon \mathcal{T}\rightarrow \mathcal{U}_{\geq n}$ the left adjoint of the inclusion $\mathcal{U}_{\geq n}\hookrightarrow \mathcal{T}$. They are called the \emph{truncation functors} associated to $\mathcal{U}$. For each object $Z\in\mathcal{T}$, there is a canonical  exact triangle
\begin{align}\label{equ:func-tri}
\tau_{\leq n}(Z)\longrightarrow Z\longrightarrow \tau_{\geq n+1}(Z)\longrightarrow \Sigma \tau_{\leq n}(Z).
\end{align}
Indeed, this triangle in the case $n=0$ is isomorphic to the one in (T3) above. The composition $H^0=\tau_{\leq 0}\tau_{\geq 0}\colon \mathcal{T}\rightarrow \mathcal{U}_0$ is the \emph{cohomological functor} associated to $\mathcal{U}$. More generally, we set $H^n(Z)=H^0\Sigma^n(Z)$, which is canonically isomorphic to $\Sigma^n\tau_{\leq n}\tau_{\geq n}(Z)$.  For details, we refer to \cite[Section~1.3]{BBD}.

A $t$-structure $\mathcal{U}$ is said to be \emph{non-degenerate} provided that
\begin{align}\label{equ:non-deg}
\bigcap_{n\in \mathbb{Z}} \mathcal{U}_{\leq n}=\{0\}=\bigcap_{n\in \mathbb{Z}} \mathcal{U}_{\geq n}.
 \end{align}
 The following lemma is standard; see \cite[the proof of Proposition~1.3.7]{BBD}.

\begin{lem}\label{lem:non-de}
Let $\mathcal{U}$ be a non-degenerate $t$-structure and $f\colon X\rightarrow Y$ be a morphism with each $H^n(f)$ an isomorphism. Then $f$ is an isomorphism.
\end{lem}

\begin{proof}
Denote by $Z$ the cone of $f$. Then the assumptions imply that $H^n(Z)\simeq 0$ for each $n$. The non-degeneration conditions imply that $Z\simeq 0$, and thus $f$ is an isomorphism.
\end{proof}

A $t$-structure $\mathcal{U}$ is \emph{bounded} provided that
$$\bigcup_{n\in \mathbb{Z}} \mathcal{U}_{\leq n}=\mathcal{T}=\bigcup_{n \in \mathbb{Z}} \mathcal{U}_{\geq n}.$$
Then $\mathcal{U}$ is necessarily non-degenerate. We observe that any object $Z\in \bigcap_{n\in \mathbb{Z}} \mathcal{U}_{\leq n}$ lies in $\mathcal{U}_{\geq m}$ for some $m$. In view of (T2), we have $\mathcal{U}_{\geq m} \cap\mathcal{U}_{\leq m-1}=\{0\}$. This forces that $Z\simeq 0$, that is, $\bigcap_{n\in \mathbb{Z}} \mathcal{U}_{\leq n}=\{0\}$. Similarly, one proves $\bigcap_{n\in \mathbb{Z}} \mathcal{U}_{\geq n}=\{0\}$.

In general, for any $t$-structure $\mathcal{U}$ and $m\leq n$, we set $\mathcal{U}_{[m, n]}=\mathcal{U}_{\geq m}\cap \mathcal{U}_{\leq n}$. By the triangles (\ref{equ:func-tri}), we observe that
$$\mathcal{U}_{[m, n]}=\Sigma^{-m}\mathcal{U}_0\ast \Sigma^{-(m+1)}\mathcal{U}_0\ast \cdots \ast \Sigma^{-n}\mathcal{U}_0.$$
We set
$$\mathcal{U}_\mathbb{N}=\bigcup_{n\geq 0} \mathcal{U}_{[-n, n]}.$$
It coincides with the smallest triangulated full subcategory containing the heart $\mathcal{U}_0$. Moreover, the $t$-structure $\mathcal{U}$ restricts to a bounded $t$-structure $$\mathcal{U}^b=(\mathcal{U}_\mathbb{N}\cap\mathcal{U}_{\leq 0}, \mathcal{U}_\mathbb{N}\cap\mathcal{U}_{\geq 0})$$
on $\mathcal{U}_\mathbb{N}$. We observe that the given $t$-structure $\mathcal{U}$ on $\mathcal{T}$ is bounded if and only if $\mathcal{T}=\mathcal{U}_\mathbb{N}$.

The following example is standard.

\begin{exm}\label{exm:standard}
Let $\mathcal{A}$ be an abelian category and $\mathbf{D}(\mathcal{A})$ its unbounded derived category. Set $\mathbf{D}(\mathcal{A})_{\leq 0}=\{X\in \mathbf{D}(\mathcal{A})\; |\; H^i(X)=0 \mbox{ for any } i>0\}$ and $\mathbf{D}(\mathcal{A})_{\geq 0}=\{X\in \mathbf{D}(\mathcal{A})\; |\; H^i(X)=0 \mbox{ for any } i<0\}$. Here, $H^i(X)$ denotes the $i$-th cohomology of $X$. Then $\mathcal{D}=(\mathbf{D}(\mathcal{A})_{\leq 0}, \mathbf{D}(\mathcal{A})_{\geq 0})$ is a non-degenerate $t$-structure on $\mathbf{D}(\mathcal{A})$, called the \emph{standard} $t$-structure. The heart $\mathcal{D}_0$ is naturally equivalent to $\mathcal{A}$, where we identify an object in $\mathcal{A}$ with the corresponding stalk complex concentrated in degree zero. Moreover, we have
$$\mathcal{D}_\mathbb{N}=\{X\in \mathbf{D}(\mathcal{A}) \; |\; H^i(X)\neq 0\mbox{ for only finitely many } i\},$$
which is equivalent to the bounded derived category $\mathbf{D}^b(\mathcal{A})$.

For a complex $X=(X^n, d^n_X)_{n\in\mathbb{Z}}$ in $\mathcal{D}(\mathcal{A})$, the truncations associated to $\mathcal{D}$ are given by
$$\tau_{\leq n}(X)= \cdots\longrightarrow X^{n-2}\longrightarrow X^{n-1} \longrightarrow {\rm Ker}d_X^{n} \longrightarrow 0$$
and
$$\tau_{\geq n}(X)= 0\longrightarrow {\rm Cok}d_X^{n-1} \longrightarrow X^{n+1}\longrightarrow X^{n+2}\longrightarrow \cdots$$
\end{exm}

Let $\mathcal{U}$ be a $t$-structure on $\mathcal{T}$. For two objects $X$ and $Y$, we have the following two canonical maps:
\begin{align}\label{equ:hom-inj1}
{\rm Hom}_\mathcal{T}(X, Y) \longrightarrow \varprojlim_{n\to +\infty} {\rm Hom}_\mathcal{T}(\tau_{\leq n}(X), Y)
\end{align}
and
\begin{align}\label{equ:hom-inj2}
{\rm Hom}_\mathcal{T}(X, Y) \longrightarrow \varprojlim_{n\to +\infty} {\rm Hom}_\mathcal{T}(X, \tau_{\geq -n}(Y)).
\end{align}

The following definition is inspired by \cite[Subsection~6.1]{Kel05}. The terminology is somehow justified by Proposition~\ref{prop:equality} below.

 \begin{defn}\label{defn:adm}
The $t$-structure $\mathcal{U}$ on $\mathcal{T}$ is called \emph{regular}, provided that (\ref{equ:hom-inj1}) is injective for any $Y\in \bigcup_{m\in \mathbb{Z}} \mathcal{U}_{\leq m}$, and that (\ref{equ:hom-inj2}) is injective for any $X\in \bigcup_{m\in \mathbb{Z}} \mathcal{U}_{\geq m}$.
\end{defn}

It is clear that a bounded $t$-structure is regular, in which case the maps (\ref{equ:hom-inj1}) and (\ref{equ:hom-inj2}) are both isomorphisms.

For any full additive subcategory $\mathcal{X}$ of $\mathcal{T}$, we denote by $[\mathcal{X}]$ the ideal formed by morphisms factoring though $\mathcal{X}$. In other words, we have $$[\mathcal{X}](A, B)=\{f\colon A\rightarrow B\; |\;  f=u\circ v \mbox{ for some }v\colon A\rightarrow X, u\colon X\rightarrow B, \mbox{ and } X\in \mathcal{X} \}.$$

The following characterization indicates some similarity between regular $t$-structures and non-degenerate ones; see (\ref{equ:non-deg}). However, it seems that there is no implication relationship between the two concepts; compare Proposition~\ref{prop:stable} and Remark~\ref{rem:stable} below.

\begin{lem}\label{lem:regular}
Let $\mathcal{U}$ be a $t$-structure. Then $\mathcal{U}$ is regular if and only if
$$\bigcap_{n\in \mathbb{Z}} [\mathcal{U}_{\geq n}](-, Y)=0=\bigcap_{n\in \mathbb{Z}} [\mathcal{U}_{\leq n}](X, -)$$
for any  $Y\in \bigcup_{m\in \mathbb{Z}} \mathcal{U}_{\leq m}$ and $X\in \bigcup_{m\in \mathbb{Z}} \mathcal{U}_{\geq m}$.
\end{lem}

\begin{proof}
Denote by $\phi$ the map in (\ref{equ:hom-inj1}). We claim that the kernel of $\phi$ equals $\bigcap_{n\in \mathbb{Z}} [\mathcal{U}_{\geq n}](X, Y)$.

We consider the canonical exact triangle  (\ref{equ:func-tri}):
$$\tau_{\leq n}(X)\stackrel{\iota_n}\longrightarrow X \longrightarrow \tau_{\geq n+1}(X)\longrightarrow \Sigma \tau_{\leq n}(X).$$
Take any $f\in \bigcap_{n\in \mathbb{Z}} [\mathcal{U}_{\geq n}](X, Y)$ and fix $n\in \mathbb{Z}$. There is a factorization $X\stackrel{v}\rightarrow U\stackrel{u}\rightarrow Y$ of $f$ for some $U\in \mathcal{U}_{\geq n+1}$. Since ${\rm Hom}_\mathcal{T}(\tau_{\leq n}(X), U)=0$, we have $v\circ \iota_n=0$. It implies that $f\circ \iota_n=0$. Letting $n$ vary, we infer  $\phi(f)=0$.

Conversely, take any map  $g\colon X\rightarrow Y$ satisfying $\phi(g)=0$. Therefore, $g\circ \iota_n=0$ for each $n$. By the exact triangle above, we infer that $g$ factors through $\tau_{\geq n+1}(X)$. This shows that $g\in [\mathcal{U}_{\geq n+1}](X, Y)$. Letting $n$ vary, we infer that $g$ lies in $\bigcap_{n\in \mathbb{Z}} [\mathcal{U}_{\geq n}](X, Y)$. The proof of the claim is completed.

By the claim above, the injectivity of (\ref{equ:hom-inj1}) is equivalent to $\bigcap_{n\in \mathbb{Z}} [\mathcal{U}_{\geq n}](-, Y)=0$ for any $Y\in \bigcup_{m\in \mathbb{Z}} \mathcal{U}_{\leq m}$. A similar result holds for (\ref{equ:hom-inj2}), yielding the required result.
\end{proof}

Following \cite[Definition~9.14]{Miy}, a $t$-structure $\mathcal{U}$ is \emph{stable}, if $\Sigma \mathcal{U}_{\leq 0}=\mathcal{U}_{\leq 0}$, or equivalently, both $\mathcal{U}_{\leq 0}$ and $\mathcal{U}_{\geq 0}$ are triangulated subcategories of $\mathcal{T}$. In this situation, $\mathcal{U}_{\leq 0}$ is a right admissible subcategory, and $\mathcal{U}_{\geq 0}$ is a left admissible subcategory  of $\mathcal{T}$; see \cite[\S 1]{BK}.

In what follows, we will see  that stable $t$-structures should be viewed as degenerate  ones.  We begin with a well-known fact.

\begin{lem}
A $t$-structure $\mathcal{U}$ is stable if and only if its heart $\mathcal{U}_0=\{0\}$.
\end{lem}

\begin{proof}
For the ``only if" part, we observe that $\mathcal{U}_{\leq -1}=\mathcal{U}_{\leq 0}$. Then  $\mathcal{U}_0=\mathcal{U}_{\leq -1}\cap \mathcal{U}_{\geq 0}=\{0\}$, where the later equality follows from (T2); see also (\ref{equ:ortho}).

For the ``if" part, we will prove that $\mathcal{U}_{\leq -1}=\mathcal{U}_{\leq 0}$. It suffices to prove that each object $X$ in $\mathcal{U}_{\leq 0}$ lies in $\mathcal{U}_{\leq -1}$. Consider the canonical exact triangle
 $$\tau_{\leq -1}(X)\longrightarrow X\longrightarrow \tau_{\geq 0}(X)\longrightarrow \Sigma \tau_{\leq -1}(X).$$
We observe that $\tau_{\geq 0}(X)$ lies in $\mathcal{U}_0$ and is zero. It follows that $\tau_{\leq -1}(X)\simeq X$, which implies that $X$ lies in $\mathcal{U}_{\leq -1}$.
\end{proof}

For two full subcategories $\mathcal{X}$ and $\mathcal{Y}$ of $\mathcal{T}$,  we write $$\mathcal{T}=\mathcal{X}\times \mathcal{Y},$$
if ${\rm Hom}_\mathcal{T}(\mathcal{X}, \mathcal{Y})=0={\rm Hom}_\mathcal{T}(\mathcal{Y}, \mathcal{X})$, and any object $Z$ in $\mathcal{T}$ admits a decomposition $Z\simeq X\oplus Y$ with some $X\in \mathcal{X}$ and $Y\in \mathcal{Y}$.

\begin{prop}\label{prop:stable}
Let $\mathcal{U}=(\mathcal{U}_{\leq 0}, \mathcal{U}_{\geq 0})$ be a  stable $t$-structure on $\mathcal{T}$. Then $\mathcal{U}$ is regular if and only if $\mathcal{T}=\mathcal{U}_{\leq 0}\times \mathcal{U}_{\geq 0}$.
\end{prop}

\begin{proof}
For the ``if" part, we claim that the map (\ref{equ:hom-inj1}) is an isomorphism. We observe that $Y$ lies in $\mathcal{U}_{\leq 0}$. If $X=X_{\leq 0}\oplus X_{\geq 0}$ with $X_{\leq 0}\in \mathcal{U}_{\leq 0}$ and $X_{\geq 0}\in \mathcal{U}_{\geq 0}$. We observe that for each $n$, $\tau_{\leq n}(X)=X_{\leq 0}$ and the natural map ${\rm Hom}_\mathcal{T}(X, Y)\rightarrow {\rm Hom}_\mathcal{T}(\tau_{\leq n}(X), Y)$ is an isomorphism. Then we infer the claim. Similarly, the map (\ref{equ:hom-inj2}) is also an isomorphism.

For the ``only if" part, we observe ${\rm Hom}_\mathcal{T}(\mathcal{U}_{\leq 0}, \mathcal{U}_{\geq 0})=0$, as $\mathcal{U}$ is stable. We claim ${\rm Hom}_\mathcal{T}(\mathcal{U}_{\geq 0}, \mathcal{U}_{\leq 0})=0$. Then the exact triangle  in (T3) will split, and the object $Z$ is isomorphic to $X\oplus \Sigma^{-1}(Y)$. This implies $\mathcal{T}=\mathcal{U}_{\leq 0}\times \mathcal{U}_{\geq 0}$.

For the claim above, take $X\in \mathcal{U}_{\geq 0}$ and $Y\in \mathcal{U}_{\leq 0}$. For each $n$, we have $\tau_{\leq n}(X)\simeq \tau_{\leq -1}(X)=0$. The injectivity of (\ref{equ:hom-inj1}) implies that ${\rm Hom}_\mathcal{T}(X, Y)=0$, as required. This completes the whole proof.
\end{proof}

\begin{rem}\label{rem:stable}
Assume that a $t$-structure  $\mathcal{U}=(\mathcal{U}_{\leq 0}, \mathcal{U}_{\geq 0})$ is both stable and non-degenerate. Then the category $\mathcal{T}$ has to be the zero category. Indeed, for any object $X$, consider the zero map $f\colon 0\rightarrow X$. As the heart $\mathcal{U}_0$ is zero, each $H^n(f)$ is trivially an isomorphism. By Lemma~\ref{lem:non-de}, $f$ is an isomorphism, as required.
\end{rem}

\section{Hereditary t-structures and global dimensions}

In this section, we study hereditary $t$-structures and the global dimension of a $t$-structure in general.

Let $\mathcal{T}$ be a triangulated category. Following \cite[Subsection~6.1]{Kel05}, a $t$-structure $\mathcal{U}=(\mathcal{U}_{\leq 0}, \mathcal{U}_{\geq 0})$ on $\mathcal{T}$ is said to be \emph{hereditary}, provided that ${\rm Hom}_\mathcal{T}(\mathcal{U}_{\geq 0}, \mathcal{U}_{\leq -2})=0$. In this case, the restricted $t$-structure $\mathcal{U}^b$ on  $\mathcal{U}_\mathbb{N}$ is necessarily hereditary. Moreover, the common heart $\mathcal{U}_0$ is a hereditary abelian category and each object $Z\in \mathcal{U}_\mathbb{N}$ is non-canonically isomorphic to $\bigoplus_{n\in \mathbb{Z}} \Sigma^{-p}H^p(Z)$; for a proof, we refer to \cite[Lemma~2.1]{CRin}. We mention that hereditary $t$-structures are said to be split in \cite[Section~2]{Hub}.

\begin{lem}
A $t$-structure $\mathcal{U}$ is hereditary if and only if each object $Z$  is isomorphic to $X\oplus Y$ for some $X\in \mathcal{U}_{\leq 0} $ and $Y\in \mathcal{U}_{\geq 1}$.
\end{lem}

\begin{proof}
The ``only if" part follows from the triangle in (T3), as the rightmost morphism is zero. For the ``if" part, we take a morphism $u\colon A\rightarrow B$ with $A\in \mathcal{U}_{\geq 0}$ and $B\in \mathcal{U}_{\leq -2}$. Using the translation functor $\Sigma^2$, we infer that the cone of $u$ has a decomposition ${\rm Cone}(u)=C_1\oplus C_2$ with $C_1\in \mathcal{U}_{\leq -2}$ and $C_2\in \mathcal{U}_{\geq -1}$. So, we have an exact triangle
$$A \stackrel{u}\longrightarrow B\stackrel{\binom{v}{0}}\longrightarrow C_1\oplus C_2\longrightarrow \Sigma(A).$$
Here, we use the fact ${\rm Hom}_\mathcal{T}(B, C_2)=0$; see (T2). Moreover, we observe that ${\rm Hom}_\mathcal{T}(C_1, \Sigma(A))=0$ and ${\rm Hom}_\mathcal{T}(B, A)=0$. Then \cite[Lemma~2.4(3)]{CRin} implies that $u=0$.
\end{proof}

The following result, extending the well-known one in \cite[Subsection~1.6]{Kra}, is expected in \cite[Proposition~1(b)]{Kel05}.

\begin{prop}\label{prop:hereditary}
Let $\mathcal{U}$ be a hereditary $t$-structure on $\mathcal{T}$. Assume further that $\mathcal{U}$ is non-degenerate satisfying the following condition: for any  family $\{A_n\;|\; n\in \mathbb{Z}\}$ of objects in $\mathcal{U}_0$, the coproduct $\coprod_{n\in \mathbb{Z}} \Sigma^{-n} A_n$ and product $\prod_{n\in \mathbb{Z}} \Sigma^{-n} A_n$ exist in $\mathcal{T}$, and the canonical morphism
$$\coprod_{n\in \mathbb{Z}} \Sigma^{-n} A_n\longrightarrow \prod_{n\in \mathbb{Z}} \Sigma^{-n} A_n$$
is an isomorphism. Then each object $Z$ in $\mathcal{T}$ is isomorphic to $\coprod_{n\in \mathbb{Z}} \Sigma^{-n} H^n(Z)$.
\end{prop}

\begin{proof}
We first claim that the canonical map $\Sigma^{-p}A_p\hookrightarrow \coprod_{n\in \mathbb{Z}} \Sigma^{-n} A_n$ induces an isomorphism on their $p$-th cohomology $H^p$. For this,  we write
$$\coprod_{n\in \mathbb{Z}} \Sigma^{-n} A_n=\coprod_{n<p} \Sigma^{-n} A_n\oplus \Sigma^{-p}A_p\oplus \prod_{n>p} \Sigma^{-n} A_n:=L\oplus \Sigma^{-p}A_p \oplus R.$$
In view of (\ref{equ:ortho}), we infer that $\mathcal{U}_{\leq p-1}$ is closed under coproducts, and that $\mathcal{U}_{\geq p+1}$ is closed under products. Therefore, $L$ lies in $\mathcal{U}_{\leq p-1}$ and $R$ lies in $\mathcal{U}_{\geq p+1}$. We conclude that $H^p(L)\simeq 0\simeq H^p(R)$. This implies the claim.

For any object $Z$ and any integer $p$, we consider the canonical exact triangle
$$\tau_{\leq p-1}(Z)\longrightarrow \tau_{\leq p}(Z)\stackrel{a_p} \longrightarrow \Sigma^{-p}H^p(Z)\longrightarrow \Sigma \tau_{\leq p-1}(Z).$$
As the rightmost morphism vanishes, the middle morphism $a_p$ admits a section, say $\iota_p\colon \Sigma^{-p}H^p(Z)\rightarrow \tau_{\leq p}(Z)$. Consider the following composite map
$$s_p\colon \Sigma^{-p}H^p(Z) \stackrel{\iota_p}\longrightarrow \tau_{\leq p}(Z)\longrightarrow Z,$$
where the unnamed arrow is the canonical morphism $\tau_{\leq p}(Z)\rightarrow Z$ in (\ref{equ:func-tri}). We observe that each $H^p(s_p)$ is an isomorphism.

These morphisms $s_p$ uniquely induce a morphism
$$s\colon \coprod_{p\in \mathbb{Z}} \Sigma^{-p} H^p(Z)\longrightarrow Z.$$
By the claim above, we infer that each $H^p(s)$ is an isomorphism. By Lemma~\ref{lem:non-de}, we infer that $s$ is an isomorphism, as required.
\end{proof}

The above consideration motivates the following concept; compare \cite[Subsection~3.1]{Q} and \cite[Subsection~1.1]{Lor}.

\begin{defn}
Let $\mathcal{U}$ be a $t$-structure on $\mathcal{T}$. The \emph{global dimension} of $\mathcal{U}$, denoted by ${\rm gl.dim}(\mathcal{U})$, is the smallest natural number $d$ such that
$${\rm Hom}_\mathcal{T}(\mathcal{U}_{\geq 0}, \mathcal{U}_{\leq -(d+1)})=0.$$
If such $d$ does not exist, we set ${\rm gl.dim}(\mathcal{U})=+\infty$.
\end{defn}

Therefore, $\mathcal{U}$ is hereditary if and only if ${\rm gl.dim}(\mathcal{U})\leq 1$. We observe
$${\rm gl.dim}(\mathcal{U}^b)\leq {\rm gl.dim}(\mathcal{U}).$$
In the following degenerate example, one might have a strict inequality.

\begin{exm}\label{exm:stable}
Let $\mathcal{U}$ be a stable $t$-structure on $\mathcal{T}$. Then ${\rm gl.dim}(\mathcal{U})<+\infty$ if and only if ${\rm gl.dim}(\mathcal{U})=0$, if and only if $\mathcal{T}=\mathcal{U}_{\leq 0}\times \mathcal{U}_{\geq 0}$. The latter condition is equivalent to the one that $\mathcal{U}$ is regular; see Proposition~\ref{prop:stable}. Therefore, for a stable non-regular $t$-structure $\mathcal{U}$, we have ${\rm gl.dim}(\mathcal{U})=+\infty$ but ${\rm gl.dim}(\mathcal{U}^b)=0$.
\end{exm}

We  explore the situations where the equality occurs.

\begin{prop}\label{prop:equality}
Let $\mathcal{U}$ be a $t$-structure on $\mathcal{T}$. Consider the following conditions.
\begin{enumerate}
\item ${\rm gl.dim}(\mathcal{U})<+\infty$.
\item The $t$-structure $\mathcal{U}$ is regular.
\item ${\rm gl.dim}(\mathcal{U}^b)={\rm gl.dim}(\mathcal{U})$.
\end{enumerate}
Then we have implications ``(1)$\Rightarrow$ (2) $\Rightarrow$ (3)".
\end{prop}

As any bounded $t$-structure is regular, by Proposition~\ref{prop:gl.dim}(1) below we infer that  (2) does not imply (1) in general. For a non-bounded example, we refer to Example~\ref{exm:regular}.

\begin{proof}
The implication ``(1) $\Rightarrow$ (2)" actually follows from Lemma~\ref{lem:regular}. We give a direct proof. Assume that ${\rm gl.dim}(\mathcal{U})=d$. We only prove the injectivity of (\ref{equ:hom-inj1}). Assume that $Y\in \mathcal{U}_{\leq m}$. We apply ${\rm Hom}_\mathcal{T}(-, Y)$ to the canonical exact triangle
\begin{align*}
\tau_{\leq m+d}(X)\stackrel{\iota}\longrightarrow X\longrightarrow \tau_{\geq m+d+1}(X)\longrightarrow \Sigma \tau_{\leq m+d}(X),
\end{align*}
and observe that ${\rm Hom}_\mathcal{T}(\tau_{\geq m+d+1}(X), Y)=0$ by ${\rm gl.dim}(\mathcal{U})=d$. Then we infer that the natural map
$${\rm Hom}_\mathcal{T}(X, Y)\longrightarrow {\rm Hom}_\mathcal{T}(\tau_{\leq m+d}(X), Y),$$
sending $f$ to $\iota\circ f$, is injective.  As this map factors through the one in (\ref{equ:hom-inj1}), we deduce that the latter is injective, as required.

To show ``(2) $\Rightarrow$ (3)", we may assume that ${\rm gl.dim}(\mathcal{U}^b)=l$ is finite. We take $X\in \mathcal{U}_{\geq 0}$ and $Y\in \mathcal{U}_{\leq -(l+1)}$. By the injectivity assumptions, we have an injective map
$${\rm Hom}_\mathcal{T}(X, Y)\rightarrow \varprojlim_{n\to +\infty} {\rm Hom}_\mathcal{T}(\tau_{\leq n}(X), Y)\rightarrow \varprojlim_{n\to +\infty} \varprojlim_{m\to +\infty} {\rm Hom}_\mathcal{T}(\tau_{\leq n}(X), \tau_{\geq -m}(Y)).$$
We observe that $\tau_{\leq n}(X)\in \mathcal{U}^b_{\geq 0}$ and $\tau_{\geq -m}(Y)\in \mathcal{U}^b_{\leq -(l+1)}$. Then we deduce $ {\rm Hom}_\mathcal{T}(\tau_{\leq n}(X), \tau_{\geq -m}(Y))=0$ by ${\rm gl.dim}(\mathcal{U}^b)=l$. We conclude  ${\rm Hom}_\mathcal{T}(X, Y)=0$, as required.
\end{proof}

For an abelian category $\mathcal{A}$, its global dimension ${\rm gl.dim}(\mathcal{A})$ is defined to be the supremum of $\{d\geq 0\; |\; {\rm Ext}_\mathcal{A}^d(-, -)\neq 0\}$.

\begin{prop}\label{prop:gl.dim}
Let $\mathcal{A}$ be an abelian category. Denote by $\mathcal{D}$ the standard t-structure on $\mathcal{D}(\mathcal{A})$. Then the following statements hold.
\begin{enumerate}
\item ${\rm gl.dim}(\mathcal{A})={\rm gl.dim}(\mathcal{D}^b)\leq {\rm gl.dim}(\mathcal{D})$.
\item Assume that $\mathcal{A}$ has enough projective objects or enough injective objects. Then we have
$${\rm gl.dim}(\mathcal{A})={\rm gl.dim}(\mathcal{D}^b)= {\rm gl.dim}(\mathcal{D}).$$
\end{enumerate}
\end{prop}

\begin{proof}
The equality in (1) is well known. We just recall that ${\rm Ext}_\mathcal{A}^n(M, N)={\rm Hom}_{\mathcal{D}(\mathcal{A})}(M, \Sigma^n(N))$ for any objects $M, N$ and $n\geq 0$; moreover, $\Sigma^n(N)$ lies in $\mathcal{D}^b_{\leq -n}$.

For (2), we assume that $\mathcal{A}$ has enough projective objects. In view of (1), we may assume that ${\rm gl.dim}(\mathcal{A})=d$ is finite. It suffices to claim that ${\rm Hom}_{\mathbf{D}(\mathcal{A})}(X, Y)=0$ for any $X\in \mathbf{D}(\mathcal{A})_{\geq 0}$ and $Y\in \mathbf{D}(\mathcal{A})_{\leq -(d+1)}$. By truncations, we may assume that $X^n=0$ for $n<0$ and $Y^m=0$ for $m>-(d+1)$. By \cite[Proposition~3.4]{Chen11}, there is a complex $P$ consisting of projective objects with a quasi-isomorphism $P\rightarrow X$  satisfying $P^n=0$ for $n\leq -(d+1)$. By the same result, there is a complex $Q$ consisting of projective objects with a quasi-isomorphism $Q\rightarrow Y$ satisfying $Q^m=0$ for $m>-(d+1)$. By \cite[Proposition~3.5]{Chen11}, the natural map
$${\rm Hom}_{\mathbf{K}(\mathcal{A})}(P, Q)\longrightarrow {\rm Hom}_{\mathbf{D}(\mathcal{A})}(P, Q)$$
is an isomorphism, where $\mathbf{K}(\mathcal{A})$ denotes the homotopy category of complexes in $\mathcal{A}$.  However, it is clear that ${\rm Hom}_{\mathbf{K}(\mathcal{A})}(P, Q)=0$. This proves the claim.
\end{proof}

\begin{rem}
Proposition~\ref{prop:gl.dim}(2) yields another proof of \cite[Lemma~6]{KK}, as a Grothendieck category always has enough injective objects. We do not know whether the assumptions in Proposition~\ref{prop:gl.dim}(2) are really needed. We mention the following fact due to \cite[Theorem~A.1]{PS}: for a hereditary abelian category $\mathcal{H}$, the standard $t$-structure $\mathcal{D}$ on $\mathbf{D}(\mathcal{H})$ is always hereditary.   In this case, we have ${\rm gl.dim}(\mathcal{H})={\rm gl.dim}(\mathcal{D})\leq 1$.
\end{rem}

\begin{rem}\label{rem:hereditary}
Let $\mathcal{T}$ be a triangulated category with a hereditary $t$-structure $\mathcal{U}$, which satisfies the conditions in Proposition~\ref{prop:hereditary}.  In view of   \cite[Theorem~A.1]{PS}, we expect that an unbounded version of \cite[Corollay~1.2]{Hub} holds, that is,  the triangulated category $\mathcal{T}$ is necessarily triangle equivalent to $\mathbf{D}(\mathcal{U}_0)$.
\end{rem}
Let $R$ be a ring with unit. We denote by $R\mbox{-Mod}$ the abelian category of left $R$-modules.  Then the global dimension of $R\mbox{-Mod}$ coincides with the left global dimension of $R$.

\begin{rem}
In general, for a $t$-structure $\mathcal{U}$ on a triangulated category $\mathcal{T}$,  there is no direct relationship between ${\rm gl.dim}(\mathcal{U}_0)$ and ${\rm gl.dim}(\mathcal{U})$. For example,   the usual  Postnikov  $t$-structure on the stable homotopy category  SH has infinite global dimension, but its heart is equivalent to $\mathbb{Z}\mbox{-Mod}$ and  has global dimension one; see \cite[Section~3]{Hub} or \cite[Subsection~4.6]{Bon}.  We refer to \cite[Example~6.3]{BZ} for an explicit example of a bounded $t$-structure $\mathcal{U}$ satisfying  ${\rm gl.dim}(\mathcal{U})\leq 4$ and ${\rm gl.dim}(\mathcal{U}_0)=+\infty$.
\end{rem}

\begin{exm}
Let $R$ be a ring. By Proposition~\ref{prop:gl.dim}(2),  the standard $t$-structure on the derived category $\mathbf{D}(R\mbox{-}{\rm Mod})$ is hereditary if and only if $R$ is a left hereditary ring. This is the classical situation studied in \cite[Subsection~1.6]{Kra}.

Combining Propositions~\ref{prop:equality} and \ref{prop:gl.dim}(2), we infer that the standard $t$-structure on  $\mathbf{D}(R\mbox{-}{\rm Mod})$ is regular, provided that the left global dimension of $R$ is finite.

Recall a general fact: for two complexes $X, Y$ of $R$-modules, we have a short exact sequence
$$0\longrightarrow {{\varprojlim}^1} {\rm Hom}(\tau_{\leq n}(X), \Sigma^{-1}(Y))\longrightarrow {\rm Hom}(X, Y) \stackrel{\phi}\longrightarrow \varprojlim {\rm Hom}(\tau_{\leq n}(X), Y)\longrightarrow 0.$$
Here, $\varprojlim=\varprojlim_{n\to +\infty}$ and ${{\varprojlim}^1}$ is its first derived functor \cite[Section~3.5]{Wei}; ${\rm Hom}$ is taken in $\mathbf{D}(R\mbox{-}{\rm Mod})$;  the truncated complexes $\tau_{\leq n}(X)$ are described in Example~\ref{exm:standard} and $\phi$ is the canonical map in (\ref{equ:hom-inj1}).  It follows that if $R$ has finite left global dimension, then the map $\phi$ is an isomorphism when the cohomology of $Y$ is bounded-above. A similar fact holds for the map (\ref{equ:hom-inj2}).

For the above short exact sequence, we observe that the colimit of the inclusion sequence
$$\tau_{\leq 0}(X)\stackrel{\iota_0}\longrightarrow \tau_{\leq 1}(X)\stackrel{\iota_1} \longrightarrow \tau_{\leq 2}(X)\longrightarrow \cdots$$
is identified with $X$. There is an exact triangle
$${\coprod}_{n\geq 0}\tau_{\leq n}(X)\stackrel{\Phi} \longrightarrow {\coprod}_{n\geq 0}\tau_{\leq n}(X)\longrightarrow X \longrightarrow \Sigma({\coprod}_{n\geq 0}\tau_{\leq n}(X)),$$
where the components of $\Phi$ are determined by $\binom{{\rm Id}}{-\iota_n}\colon \tau_{\leq n}(X)\rightarrow \tau_{\leq n}(X)\oplus \tau_{\leq n+1}(X)$. Applying the cohomological functor ${\rm Hom}(-, Y)$ to this triangle, we infer the required short exact sequence immediately.
\end{exm}

We give an example of a regular $t$-structure, which is not bounded and has infinite global dimension.

\begin{exm}\label{exm:regular}
For each $a\geq 1$, we let $\mathcal{T}^a$ be a triangulated category with a $t$-structure $\mathcal{U}^a$ such that it is not bounded and ${\rm gl.dim}(\mathcal{U}^a)=a$. For example, we might take a ring $R^a$ with left global dimension $a$,  and set $\mathcal{T}^a=\mathbf{D}(R^a\mbox{-{\rm Mod}})$.

Set $\mathcal{T}=\coprod_{a\geq 1}\mathcal{T}^a$ to be its coproduct. In other words, each object in $\mathcal{T}$ is a formal sum $X=\oplus_{a\geq 1} X^a$ such that $X^a\in \mathcal{T}^a$ satisfying $X^a=0$ for all but finitely many $a$'s. Then we have a $t$-structure $\mathcal{U}=(\coprod_{a\geq 1}\mathcal{U}^a_{\leq 0}, \coprod_{a\geq 1}\mathcal{U}^a_{\geq 0})$ on $\mathcal{T}$, which is not bounded. We observe that ${\rm gl.dim}(\mathcal{U})=+\infty$. Since each $\mathcal{U}^a$ is regular, we infer that $\mathcal{U}$ is regular.
\end{exm}

\section{The distance between t-structures}

In this section, we discuss the distance between t-structures and its relation with the global dimensions. Throughout, we fix a triangulated category $\mathcal{T}$.

\begin{lem}\label{lem:two-t}
Let $\mathcal{U}=(\mathcal{U}_{\leq 0}, \mathcal{U}_{\geq 0})$ and $\mathcal{V}=(\mathcal{V}_{\leq 0}, \mathcal{V}_{\geq 0})$ be two t-structures on $\mathcal{T}$. Assume that $m$ and $n$ are integers satisfying $m\leq n$. Then the following statements are equivalent:
\begin{enumerate}
\item $\mathcal{V}_{\leq m}\subseteq \mathcal{U}_{\leq 0}\subseteq \mathcal{V}_{\leq n}$;
\item $\mathcal{U}_{\leq -n} \subseteq \mathcal{V}_{\leq 0}\subseteq \mathcal{U}_{\leq -m}$;
\item $\mathcal{V}_{\geq n}\subseteq \mathcal{U}_{\geq 0}\subseteq \mathcal{V}_{\geq m}$;
\item $\mathcal{U}_{\geq -m}\subseteq \mathcal{V}_{\geq 0}\subseteq \mathcal{U}_{\geq -n}$;
\item $\mathcal{V}_{\leq m}\subseteq \mathcal{U}_{\leq 0}$ and $\mathcal{V}_{\geq n}\subseteq \mathcal{U}_{\geq 0}$;
\item $\mathcal{U}_{\leq -n}\subseteq \mathcal{V}_{\leq 0}$ and $\mathcal{U}_{\geq -m}\subseteq \mathcal{V}_{\geq 0}$;
\end{enumerate}
Assume that both $\mathcal{U}$ and $\mathcal{V}$ are bounded. Then any of the above statements is equivalent to each of the following ones.
\begin{enumerate}
\item[(i)] $\mathcal{U}_0\subseteq \mathcal{V}_{[m, n]}$;
\item[(i)'] $\mathcal{V}_0\subseteq \mathcal{U}_{[-n, -m]}$.
\end{enumerate}
\end{lem}

\begin{proof}
Applying the translation functor and (\ref{equ:ortho}), we infer these equivalences among (1)-(6).  In the bounded cases, we only prove that (i) is equivalent to any of (1)-(6).

By (1) and (3), we infer (i) from the following identity:
$$\mathcal{U}_0=\mathcal{U}_{\leq 0}\cap \mathcal{U}_{\geq 0}\subseteq \mathcal{V}_{\leq n}\cap \mathcal{V}_{\geq m}=\mathcal{V}_{[m, n]}.$$
Conversely, we assume $\mathcal{U}_0\subseteq \mathcal{V}_{[m, n]}$. For any $a\leq b$, we recall that $$\mathcal{U}_{[a, b]}=\Sigma^{-a}\mathcal{U}_0* \Sigma^{-(a+1)}\mathcal{U}_0 *\cdots *\Sigma^{-b}\mathcal{U}_0.$$ We infer that $\mathcal{U}_{[a, b]}\subseteq \mathcal{V}_{[a+m, b+n]}$. As $\mathcal{U}_{\leq -n}=\bigcup_{a\leq -n}\mathcal{U}_{[a, -n]}$, we have $\mathcal{U}_{\leq -n}\subseteq \mathcal{V}_{\leq 0}$. Similarly, we have $\mathcal{U}_{\geq -m}\subseteq \mathcal{V}_{\geq 0}$.¡¡We infer (6), as required.
\end{proof}

The following notion is known to experts; compare \cite[Subsection~2.4]{QW} and \cite[Subsection~2.2]{MZ}.

\begin{defn}
For two $t$-structures $\mathcal{U}$ and $\mathcal{V}$ on $\mathcal{T}$, we define their \emph{distance} $d(\mathcal{U}, \mathcal{V})$ to be the smallest natural number $d$ such that $\mathcal{V}_{\leq m}\subseteq \mathcal{U}_{\leq 0}\subseteq \mathcal{V}_{\leq m+d}$ for some integer $m$. If such $d$ does not exist, we set $d(\mathcal{U}, \mathcal{V})=+\infty$.
\end{defn}

\begin{rem}
By definition, for any two $t$-structures $\mathcal{U}$ and $\mathcal{V}$ and for any integers $m,n$, we have $d(\mathcal{U}, \mathcal{V})=d(\Sigma^m(\mathcal{U}), \Sigma^n(\mathcal{V}))$.
\end{rem}

Koszul duality \cite{BGS} provides many examples of bounded $t$-structures with infinite distances.

\begin{exm}
Denote by $k[\epsilon]=k\oplus k\epsilon$ the algebra of dual numbers over a field $k$, and by $k[t]$ the polynomial algebra in one variable. They are naturally $\mathbb{Z}$-graded by means of ${\rm deg}(\epsilon)=1={\rm deg}(t)$. The corresponding categories of finitely generated $\mathbb{Z}$-graded modules are denoted by $k[\epsilon]\mbox{-{\rm gr}}$ and $k[t]\mbox{-{\rm gr}}$, respectively. By \cite[Theorem~2.12.1]{BGS}, there is a triangle equivalence
$$F\colon \mathbf{D}^b(k[\epsilon]\mbox{-{\rm gr}})\longrightarrow \mathbf{D}^b(k[t]\mbox{-{\rm gr}})$$
such that $F(\Sigma^n(k(-n)))=k[t](n)$ for any integer $n$. Here, for any graded module $V$, its degree-shift $V(n)$ means the same module with the grading given by $V(n)_p=V_{n+p}$.

Recall that $\mathcal{D}$ denotes the standard $t$-structure on $\mathbf{D}^b(k[\epsilon]\mbox{-{\rm gr}})$. Denote by $\mathcal{D}_F$ the $t$-structure on $\mathbf{D}^b(k[\epsilon]\mbox{-{\rm gr}})$,  which is transported from the standard $t$-structure of $\mathbf{D}^b(k[t]\mbox{-{\rm gr}})$ via $F$.  We observe that the heart of $\mathcal{D}_F$ does not belong to $\mathcal{D}_{[-m, m]}$ for any $m$. This implies that $d(\mathcal{D}, \mathcal{D}_F)=+\infty$.
\end{exm}

\begin{exm}\label{exm:length}
Let $\mathcal{U}$ be a bounded $t$-structure $\mathcal{T}$ such that its heart $\mathcal{U}_0$ is  a length abelian category with only finitely many simple objects up to isomorphism. Such a $t$-structure is called \emph{algebraic} in \cite[Subsection~2.4]{QW}. Then $d(\mathcal{U}, \mathcal{V})<+\infty$ for any bounded $t$-structure $\mathcal{V}$ on $\mathcal{T}$.

Indeed, we take a complete set $\{S_1, S_2, \cdots, S_n\}$ of representatives of simple objects in $\mathcal{U}_0$. Then any object in $\mathcal{U}_0$ is an iterated extension of these simple objects.   There are integers $m\leq n$ such that each $S_i$ lies in $\mathcal{V}_{[m, n]}$. As $\mathcal{V}_{[m, n]}$ is closed under extensions, we infer that $\mathcal{U}_0\subseteq \mathcal{V}_{[m, n]}$. By Lemma~\ref{lem:two-t}, we have $\mathcal{V}_{\leq m}\subseteq \mathcal{U}_{\leq 0}\subseteq \mathcal{V}_{\leq n}$, which implies $d(\mathcal{U}, \mathcal{V})\leq n-m<+\infty$.
\end{exm}

\begin{lem}\label{lem:dist}
Keep the notation as above. Then the following statements hold.
\begin{enumerate}
\item $d(\mathcal{U}, \mathcal{V})=d(\mathcal{V}, \mathcal{U})$.
\item $d(\mathcal{U}, \mathcal{V})=0$ if and only if $\mathcal{V}=\Sigma^m(\mathcal{U})$ for some  integer $m$.
\item For a third t-structure $\mathcal{W}$, we have $d(\mathcal{U}, \mathcal{W})\leq d(\mathcal{U}, \mathcal{V})+d(\mathcal{V}, \mathcal{W})$.
\item ${\rm gl.dim}(\mathcal{U})\leq {\rm gl.dim}(\mathcal{V})+d(\mathcal{U}, \mathcal{V})$ and ${\rm gl.dim}(\mathcal{V})\leq {\rm gl.dim}(\mathcal{U})+d(\mathcal{U}, \mathcal{V})$.
\end{enumerate}
\end{lem}

\begin{proof}
(1) follows from the equivalence ``(1) $\Leftrightarrow$ (2)" of Lemma~\ref{lem:two-t}. For (2), it suffices to observe that $d(\mathcal{U}, \mathcal{V})=0$ if and only if $\mathcal{U}_{\leq 0}=\mathcal{V}_{\leq m}$ for some $m$.

For (3), we may assume that $d(\mathcal{U}, \mathcal{V})=d_1$ and $d(\mathcal{V}, \mathcal{W})=d_2$ are natural numbers. Then we have $\mathcal{V}_{\leq m_1}\subseteq \mathcal{U}_{\leq 0}\subseteq  \mathcal{V}_{\leq m_1+d_1}$ and   $\mathcal{W}_{\leq m_2}\subseteq \mathcal{V}_{\leq 0}\subseteq \mathcal{W}_{\leq m_2+d_2}$.  From the latter inequality and taking translations, we infer that $\mathcal{W}_{\leq m_1+m_2}\subseteq \mathcal{V}_{\leq m_1}$ and $\mathcal{V}_{\leq m_1+d_1}\subseteq \mathcal{W}_{\leq m_1+m_2+(d_1+d_2)}$. So we have
$$\mathcal{W}_{\leq m_1+m_2}\subseteq \mathcal{U}_{\leq 0}\subseteq \mathcal{W}_{\leq m_1+m_2+(d_1+d_2)}.$$
This proves that $d(\mathcal{U}, \mathcal{W})\leq d_1+d_2$, the required triangle inequality.

To prove (4),  we only prove the left inequality. We  may assume that both $d(\mathcal{U}, \mathcal{V})=d$ and ${\rm gl.dim}(\mathcal{V})=g$ are finite.  We assume further that $\mathcal{V}_{\leq m}\subseteq \mathcal{U}_{\leq 0}\subseteq \mathcal{V}_{\leq m+d}$. By Lemma~\ref{lem:two-t}, we obtain $\mathcal{U}_{\geq -m}\subseteq \mathcal{V}_{\geq 0}$ and $\mathcal{U}_{\leq -(g+m+d+1)}\subseteq \mathcal{V}_{\leq -(g+1)}$. By ${\rm gl.dim}(\mathcal{V})=g$, we have ${\rm Hom}_\mathcal{T}(\mathcal{V}_{\geq 0}, \mathcal{V}_{\leq -(g+1)})=0$. Then we infer that
$${\rm Hom}_\mathcal{T}(\mathcal{U}_{\geq -m}, \mathcal{U}_{\leq -(g+m+d+1)})=0.$$
By translations, this is equivalent to ${\rm Hom}_\mathcal{T}(\mathcal{U}_{\geq 0}, \mathcal{U}_{\leq -(g+d+1)})=0$, which implies that ${\rm gl.dim}(\mathcal{U})\leq g+d$.
\end{proof}

The following main result of this section shows that two $t$-structures with a finite distance share many properties.

\begin{prop}\label{prop:dist-finite}
Let $\mathcal{U}$ and $\mathcal{V}$ be two $t$-structures in $\mathcal{T}$ satisfying $d(\mathcal{U}, \mathcal{V})<+\infty$. Then the following statements hold.
\begin{enumerate}
\item $\bigcap_{n\in \mathbb{Z}}\mathcal{U}_{\leq n}= \bigcap_{n\in \mathbb{Z}}\mathcal{V}_{\leq n}$ and $\bigcap_{n\in \mathbb{Z}}\mathcal{U}_{\geq n}= \bigcap_{n\in \mathbb{Z}}\mathcal{V}_{\geq n}$. Therefore, $\mathcal{U}$ is non-degenerate if and only if so is $\mathcal{V}$.
\item $\mathcal{U}_\mathbb{N}=\mathcal{V}_\mathbb{N}$. Therefore, $\mathcal{U}$ is bounded if and only if so is $\mathcal{V}$.
\item The $t$-structure $\mathcal{U}$ is regular if and only if so is $\mathcal{V}$.
\item ${\rm gl.dim}(\mathcal{U})<+\infty$ if and only if ${\rm gl.dim}(\mathcal{V})<+\infty$.
\item Assume that either $\mathcal{U}$ or $\mathcal{V}$ is stable. Then $\mathcal{U}=\mathcal{V}$.
\end{enumerate}
\end{prop}

\begin{proof}
Assume that $d(\mathcal{U}, \mathcal{V})=d$. Take an integer $N$ such that $\mathcal{V}_{\leq N}\subseteq \mathcal{U}_{\leq 0}\subseteq \mathcal{V}_{\leq N+d}$.

We observe $\mathcal{U}_{\leq n}\subseteq \mathcal{V}_{\leq n+N+d}$ and $\mathcal{V}_{\leq n}\subseteq \mathcal{U}_{\leq n-N}$ for each $n$. From these inequalities, we deduce $\bigcap_{n\in \mathbb{Z}}\mathcal{U}_{\leq n}= \bigcap_{n\in \mathbb{Z}}\mathcal{V}_{\leq n}$. The remaining equality in (1) is proved similarly.

For (2), first we observe $\mathcal{V}_{\geq N+d}\subseteq \mathcal{U}_{\geq 0}\subseteq \mathcal{V}_{\geq N}$ by Lemma~\ref{lem:two-t}. Then we have $\mathcal{U}_{[m, n]}\subseteq \mathcal{V}_{[m+N, n+N+d]}$, which implies that
$\mathcal{U}_\mathbb{N}\subseteq \mathcal{V}_\mathbb{N}$. By duality, we have $\mathcal{V}_\mathbb{N}\subseteq \mathcal{U}_\mathbb{N}$. For the second half, we just recall that $\mathcal{U}$ (\emph{resp}. $\mathcal{V}$ ) is bounded if and only if $\mathcal{U}_\mathbb{N}=\mathcal{T}$ (\emph{resp}. $\mathcal{V}_\mathbb{N}=\mathcal{T}$).

For (3), we only prove the ``only if" part. Assume that $\mathcal{U}$ is regular.  Take any $Y\in \bigcup_{m\in \mathbb{Z}} \mathcal{V}_{\leq m}$. By $\mathcal{V}_{\leq m}\subseteq \mathcal{U}_{\leq m-N}$, we infer that $Y$ belongs to $\bigcup_{m\in \mathbb{Z}} \mathcal{U}_{\leq m}$. For each $n\in \mathbb{Z}$, we have $\mathcal{V}_{\geq n}\subseteq \mathcal{U}_{\geq n-N-d}$. This implies the inclusion  $[\mathcal{V}_{\geq n}]\subseteq [\mathcal{U}_{\geq n-N-d}]$ of ideals. Therefore, we have
$$\bigcap_{n\in\mathbb{Z}}[\mathcal{V}_{\geq n}](-, Y)\subseteq \bigcap_{n\in\mathbb{Z}}[\mathcal{U}_{\geq n}](-, Y)=0.$$
Here, by Lemma~\ref{lem:regular}, the equality on the right hand side follows from the regularity of $\mathcal{U}$. Similarly, one proves that $\bigcap_{n\in \mathbb{Z}} [\mathcal{V}_{\leq n}](-, X)=0$ for any $X\in \bigcup_{m\in \mathbb{Z}} \mathcal{V}_{\geq m}$. By Lemma~\ref{lem:regular}, we infer that $\mathcal{V}$ is regular.

The statement (4) follows immediately from Lemma~\ref{lem:dist}(4). To prove (5), we may assume that $\mathcal{V}$ is stable. Then we have $\mathcal{V}_{\leq N}=\mathcal{V}_{\leq N+d}$, both of which equal $\mathcal{V}_{\leq 0}$. This implies that $\mathcal{U}_{\leq 0}=\mathcal{V}_{\leq 0}$. In view of (\ref{equ:ortho}), we have $\mathcal{U}=\mathcal{V}$.
\end{proof}

By Proposition~\ref{prop:dist-finite}(2), we observe that $\mathcal{U}^b$ and $\mathcal{V}^b$ are both bounded $t$-structures on the same subcategory $\mathcal{U}_\mathbb{N}=\mathcal{V}_\mathbb{N}$ of $\mathcal{T}$. Therefore, we can consider their distance.

\begin{prop}\label{prop:dist-bounded}
Let $\mathcal{U}$ and $\mathcal{V}$ be two t-structures in $\mathcal{T}$ satisfying $d(\mathcal{U}, \mathcal{V})<+\infty$. Then we have $d(\mathcal{U}, \mathcal{V})=d(\mathcal{U}^b, \mathcal{V}^b)$.
\end{prop}

\begin{proof}
Assume that $d(\mathcal{U}, \mathcal{V})=d$ and $d(\mathcal{U}^b, \mathcal{V}^b)=d'$. It is clear that $d'\leq d$. We take integers $N$ and $N'$ such that  $\mathcal{V}_{\leq N}\subseteq \mathcal{U}_{\leq 0}\subseteq \mathcal{V}_{\leq N+d}$ and  $\mathcal{V}^b_{\leq N'}\subseteq \mathcal{U}^b_{\leq 0}\subseteq \mathcal{V}^b_{\leq N'+d'}$. We denote by $\tau$ and $\sigma$ the truncations corresponding to $\mathcal{U}$ and $\mathcal{V}$, respectively.

We claim that $\mathcal{V}_{\leq N'}\subseteq \mathcal{U}_{\leq 0}$. Take any object $X\in \mathcal{V}_{\leq N'}$ and consider the canonical exact triangle
$$\sigma_{\leq N}(X)\longrightarrow X\longrightarrow \sigma_{\geq N+1}(X)\longrightarrow \Sigma\sigma_{\leq N}(X).$$
We observe that $\sigma_{\leq N}(X)\in \mathcal{V}_{\leq N}\subseteq \mathcal{U}_{\leq 0}$ and $\sigma_{\geq N+1}(X)\in \mathcal{V}^b_{\leq N'}\subseteq \mathcal{U}^b_{\leq 0}$. As $\mathcal{U}_{\leq 0}$ is closed under extensions, we infer that $X$ lies in $\mathcal{U}_{\leq 0}$, as required.

We claim that $\mathcal{U}_{\leq 0}\subseteq \mathcal{V}_{\leq N'+d'}$. This will imply that $d\leq d'$. Take any object $Y\in \mathcal{U}_{\leq 0}$ and consider the canonical exact triangle
$$\tau_{\leq n}(Y)\longrightarrow Y\longrightarrow \tau_{\geq n+1}(Y)\longrightarrow \Sigma\tau_{\leq n}(Y),$$
where $n=(N'+d')-(N+d)$. We observe that $\tau_{\leq n}(Y)\in \mathcal{U}_{\leq n}\subseteq \mathcal{V}_{\leq N'+d'}$ and $\tau_{\geq n+1}(Y)\in \mathcal{U}^b_{\leq 0}\subseteq \mathcal{V}^b_{\leq N'+d'}$. As $\mathcal{V}_{\leq N'+d'}$ is closed under extensions, we infer that $Y$ lies in it, and thus prove the claim.
\end{proof}

\section{The extensions of t-structures}

Let $\mathcal{T}$ be a triangulated category. We fix a $t$-structure $\mathcal{U}=(\mathcal{U}_{\leq 0}, \mathcal{U}_{\geq 0})$. Write $\mathcal{T}^b=\mathcal{U}_\mathbb{N}$. Then $\mathcal{U}^b$ is a bounded $t$-structure on $\mathcal{T}^b$. We emphasize that $\mathcal{T}^b$ depends on the fixed $t$-structure $\mathcal{U}$.

We reformulate the results in \cite[Subsection~6.1]{Kel05} as a bijective correspondence. We point out that a similar consideration appears implicitly in the last paragraph of \cite[2.2.1]{BBD}.

\begin{thm}[Keller]\label{thm:ext}
There is a bijective correspondence
$$\left\{
\begin{aligned}
& \mbox{t-structures } \mathcal{V} \mbox{ on }  \mathcal{T} \\
& \mbox{with } d(\mathcal{U}, \mathcal{V})<+\infty
 \end{aligned} \right\} \longrightarrow \left
 \{
 \begin{aligned}
 & \mbox{t-structures } \mathcal{V}'   \mbox{ on }   \mathcal{T}^b
  \\
  &\mbox{with } d(\mathcal{U}^b, \mathcal{V}')<+\infty
  \end{aligned}
  \right \},
 $$
 sending $\mathcal{V}$ to its restriction $\mathcal{V}^b$ on $\mathcal{T}^b$.
\end{thm}

Following \cite[Subsection~6.1]{Kel05}, we might view $\mathcal{V}$ as a \emph{canonical extension} of the bounded $t$-structure $\mathcal{V}^b$ on $\mathcal{T}^b$.

\begin{proof}
The above map is well defined, since by Proposition~\ref{prop:dist-finite}(2) $\mathcal{V}_\mathbb{N}=\mathcal{T}^b$.  We may assume that $\mathcal{U}$ is non-stable, or equivalently, $\mathcal{T}^b$ is non-trivial. Otherwise, the bijection is trivial; see Proposition~\ref{prop:dist-finite}(5).

To prove the injectivity, we take two $t$-structures $\mathcal{V}$ and $\mathcal{V}'$. Then $d(\mathcal{V}, \mathcal{V}')=d$ is finite.  Assume that $\mathcal{V}^b={\mathcal{V}'}^b$. By Proposition~\ref{prop:dist-bounded}, we have $d=0$, that is, $\mathcal{V}=\Sigma^m(\mathcal{V}')$ for some integer $m$; see Lemma~\ref{lem:dist}(2). It follows that $\mathcal{V}^b=\Sigma^m(\mathcal{V}^b)$. As $\mathcal{V}^b$ is a bounded $t$-structure, it is not stable. Therefore, we infer that $m=0$ and thus $\mathcal{V}=\mathcal{V}'$.

It remains to prove the surjectivity. We will prove that any $t$-structure $\mathcal{V}'$ with $d(\mathcal{U}^b, \mathcal{V}')=d<+\infty$ can be extended to $\mathcal{T}$. This result is stated in \cite[Proposition~1~a)]{Kel05}.

 We define $\mathcal{V}_{\leq 0}=\mathcal{U}_{\leq -(m+d)}*\mathcal{V}'_{\leq 0}$ and $\mathcal{V}_{\geq 0}=\mathcal{V}'_{\geq 0}* \mathcal{U}_{\geq -m}$. We claim that $(\mathcal{V}_{\leq 0}, \mathcal{V}_{\geq 0})$ is a $t$-structure on $\mathcal{T}$.

For the claim, we verify the defining conditions (T1)-(T3).  We first observe $\Sigma \mathcal{V}_{\leq 0}\subseteq \mathcal{V}_{\leq 0}$ by the facts  $\Sigma \mathcal{U}_{\leq -(m+d)}\subseteq \mathcal{U}_{\leq -(m+d)}$ and $\Sigma \mathcal{V}'_{\leq 0}\subseteq \mathcal{V}'_{\leq 0}$. Similarly, we have $\Sigma^{-1}\mathcal{V}_{\geq 0}\subseteq \mathcal{V}_{\geq 0}$, proving (T1). For (T2), we have $\mathcal{V}_{\geq 1}=\mathcal{V}'_{\geq 1}* \mathcal{U}_{\geq 1-m}$. We have
$${\rm Hom}_\mathcal{T}(\mathcal{U}_{\leq -(m+d)}, \mathcal{V}'_{\geq 1})=0={\rm Hom}_\mathcal{T}(\mathcal{V}'_{\leq 0}, \mathcal{U}_{\geq 1-m})$$
by using $\mathcal{V}'_{\geq 1}\subseteq \mathcal{U}^b_{\geq 1-m-d}$ and $\mathcal{V}'_{\leq 0}\subseteq \mathcal{U}^b_{-m}$, respectively. Now, the required identity ${\rm Hom}_\mathcal{T}(\mathcal{V}_{\leq 0}, \mathcal{V}_{\geq 1})=0$ follows from the following one
$${\rm Hom}_\mathcal{T}(\mathcal{V}'_{\leq 0}, \mathcal{V}'_{\geq 1})=0={\rm Hom}_\mathcal{T}(\mathcal{U}_{\leq -(m+d)}, \mathcal{U}_{\geq 1-m}).$$
To prove (T3), we observe by the $t$-structure $\mathcal{U}$ the first equality in the following inclusion.
\begin{align*}
\mathcal{T}&=\mathcal{U}_{\leq -(m+d)} * \mathcal{U}_{[1-(m+d), -m]} * \mathcal{U}_{\geq 1-m}\\
           &\subseteq \mathcal{U}_{\leq -(m+d)} * (\mathcal{V}'_{\leq 0}*\mathcal{V}'_{\geq 1}) * \mathcal{U}_{\geq 1-m}\\
           &=\mathcal{V}_{\leq 0}*\mathcal{V}_{\geq 1}
\end{align*}
Here, the middle inclusion uses the facts that $\mathcal{U}_{[1-(m+d), -m]} \subseteq \mathcal{T}^b$ and that $\mathcal{V}'$ is a $t$-structure on $\mathcal{T}^b$.   The above  inclusion clearly implies  the required (T3).

By the very definition, we have $\mathcal{U}_{\leq -(m+d)}\subseteq \mathcal{V}_{\leq 0}$ and $\mathcal{U}_{\geq -m} \subseteq \mathcal{V}_{\geq 0}$. By Lemma~\ref{lem:two-t}, we have $d(\mathcal{U}, \mathcal{V})$ is finite. Since  $\mathcal{V}'_{\leq 0}\subseteq \mathcal{V}^b_{\leq 0}$ and $\mathcal{V}'_{\geq 0}\subseteq \mathcal{V}^b_{\geq 0}$, we infer  $\mathcal{V}^b=\mathcal{V}'$  from  Lemma~\ref{lem:two-t}. This completes the whole proof.
\end{proof}

\begin{rem}
In view of Proposition~\ref{prop:dist-bounded}, the above bijection restricts to the following one: for each integer $d\geq 0$, the assignment $\mathcal{V}\mapsto \mathcal{V}^b$ yields a bijection
$$\left\{
\begin{aligned}
& t\mbox{-structures } \mathcal{V} \mbox{ on }  \mathcal{T} \\
& \mbox{with } d(\mathcal{U}, \mathcal{V})=d
 \end{aligned} \right\} \longrightarrow \left
 \{
 \begin{aligned}
 & t\mbox{-structures } \mathcal{V}'   \mbox{ on }   \mathcal{T}^b
  \\
  &\mbox{with } d(\mathcal{U}^b, \mathcal{V}')=d
  \end{aligned}
  \right \}.
 $$
 The case that $d=1$ is of particular interest. Indeed, by \cite[Proposition~2.4]{QW}, we have a bijection
 $$ \left
 \{
 \begin{aligned}
 & t\mbox{-structures } \mathcal{V}'   \mbox{ on }   \mathcal{T}^b
  \\
  &\mbox{with } d(\mathcal{U}^b, \mathcal{V}')=1
  \end{aligned}
  \right \}
   \longrightarrow
 \mathbb{Z}\times \{(\mathcal{X}, \mathcal{Y})\; |\; \mbox{ non-trivial torsion pairs in } \mathcal{U}_0\},$$
 which sends $\mathcal{V'}$ to $(m, (\mathcal{V}'_{\leq -(m+1)}\cap \mathcal{U}_0, \mathcal{V}'_{\geq -m}\cap \mathcal{U}_0))$, where $m$ is uniquely determined by  $\mathcal{U}^b_{\leq m}\subseteq \mathcal{V}'_{\leq 0}\subseteq \mathcal{U}^b_{\leq m+1}$; compare \cite[Chapter~I, Theorem~3.1]{BR}. Here, the non-triviality of $(\mathcal{X}, \mathcal{Y})$ means $\mathcal{X}\neq 0\neq \mathcal{Y}$.
\end{rem}

The main concern of \cite[Subsection~6.1]{Kel05} is hereditary $t$-structures. However, in the correspondence above, the hereditariness of $\mathcal{V}^b$ does not imply the one of $\mathcal{V}$, in general; see the following degenerate example.

\begin{exm}
Assume that $\mathcal{U}$ is a stable non-regular $t$-structure. Then $\mathcal{U}^b$ is hereditary, as the category  $\mathcal{T}^b$ is zero. However, $\mathcal{U}$ is not hereditary; see Example~\ref{exm:stable}.
\end{exm}

The following result reformulates \cite[Proposition~1~b)]{Kel05}.

\begin{prop}[Keller]\label{prop:t-heredi}
Assume that the fixed $t$-structure $\mathcal{U}$ on $\mathcal{T}$ is regular. Then the bijection in Theorem~\ref{thm:ext} restricts to a bijection
$$\left\{
\begin{aligned}
& \mbox{hereditary t-structures } \mathcal{V} \mbox{ on }  \mathcal{T} \\
& \mbox{with } d(\mathcal{U}, \mathcal{V})<+\infty
 \end{aligned} \right\} \longrightarrow \left
 \{
 \begin{aligned}
 & \mbox{hereditary t-structures } \mathcal{V}'   \mbox{ on }   \mathcal{T}^b
  \\
  &\mbox{with } d(\mathcal{U}^b, \mathcal{V}')<+\infty
  \end{aligned}
  \right \}.
 $$
\end{prop}

\begin{proof}
By Proposition~\ref{prop:dist-finite}(3), any  $t$-structure $\mathcal{V}$ in consideration is regular. Applying Proposition~\ref{prop:equality}, we infer that $\mathcal{V}$ is hereditary if and only if so is $\mathcal{V}^b$. Then we obtain the restricted bijection.
\end{proof}

\begin{rem}
We observe if the two sets in the bijection in Proposition~\ref{prop:t-heredi} are non-empty, then by Proposition~\ref{prop:dist-finite}(4) the fixed $t$-structure $\mathcal{U}$ necessarily satisfies ${\rm gl.dim}(\mathcal{U})<+\infty$.
\end{rem}

In view of Example~\ref{exm:length}, we have the following more special case.

\begin{cor}\label{cor:length}
Assume that the fixed $t$-structure $\mathcal{U}$  is regular such that its heart $\mathcal{U}_0$ is a length abelian category with only finitely many simple objects.  Then the bijection in Theorem~\ref{thm:ext} restricts to a bijection
$$\left\{
\begin{aligned}
& \mbox{hereditary t-structures } \mathcal{V} \mbox{ on }  \mathcal{T} \\
& \mbox{with } d(\mathcal{U}, \mathcal{V})<+\infty
 \end{aligned} \right\} \longrightarrow \left
 \{
 \begin{aligned}
  &\mbox{hereditary bounded t-structures}\\
   & \mathcal{V}'   \mbox{ on }   \mathcal{T}^b
  \end{aligned}
  \right \}.
 $$
\end{cor}

In the final example, we discuss the situation in \cite[Section~6]{Kel05}, which occurs naturally in the study of the triangulated structure of the cluster category.

\begin{exm}
Let $A$ be a finite dimensional algebra over a field $k$ with finite global dimension. Denote by $A\mbox{-{\rm mod}}$ the category of finite dimensional left $A$-modules and by $\mathbf{D}(A\mbox{-}{\rm mod})$ its unbounded derived category.

By \cite[Corollary~4.5]{PS}, $\mathbf{D}(A\mbox{-}{\rm mod})$  is viewed as a triangulated subcategory of $\mathbf{D}(A\mbox{-}{\rm Mod})$ formed by complexes with componentwise finite dimensional cohomology. It follows from Proposition~\ref{prop:gl.dim} that the standard $t$-structure $\mathcal{D}$ on $\mathbf{D}(A\mbox{-}{\rm mod})$ has finite global dimension.

Corollary~\ref{cor:length} applies well in this situation. We obtain a bijection
$$\left\{
\begin{aligned}
& \mbox{hereditary t-structures } \mathcal{V} \mbox{ on }  \mathbf{D}(A\mbox{-}{\rm mod}) \\
& \mbox{with } d(\mathcal{D}, \mathcal{V})<+\infty
 \end{aligned} \right\} \longrightarrow \left
 \{
 \begin{aligned}
  &\mbox{hereditary bounded t-structures}\\
   & \mathcal{V}'   \mbox{ on }   \mathbf{D}^b(A\mbox{-}{\rm mod})
  \end{aligned}
  \right \}.
 $$
 In this bijection, the $t$-structure $\mathcal{V}$ is viewed as an extension of $\mathcal{V}'$. For a fixed $\mathcal{V}'$, we denote its heart by $\mathcal{H}$, which is a hereditary abelian category. We observe that the conditions in Proposition~\ref{prop:hereditary} are satisfied for the extended $t$-structure $\mathcal{V}$. It follows that each object $X$ in $\mathbf{D}(A\mbox{-}{\rm mod})$ has a decomposition
 $$X=\bigoplus_{n\in \mathbb{Z}}\Sigma^{-n}(H_n)$$
  for some objects $H_n\in \mathcal{H}$.  This decomposition  plays a role in \cite[the proof in Section~6]{Kel05}. As mentioned in Remark~\ref{rem:hereditary}, we expect that $\mathbf{D}(A\mbox{-}{\rm mod})$ is triangle equivalent to $\mathbf{D}(\mathcal{H})$.
\end{exm}

\vskip 5pt

\noindent {\bf Acknowledgements.}\quad  This paper  was partly  written up, when the second author was visiting Hefei from September 2013 to June 2014. The second author would like to thank the first author for warm hospitality during his stay in Hefei.
This work is supported by National Natural Science Foundation of China (No.s 11901551, 11971449, and 12131015), and the Natural Science Foundation of Fujian Province (No. 2020J01075).

\bibliography{}

\vskip 10pt

 {\footnotesize \noindent Xiao-Wu Chen \\
 Key Laboratory of Wu Wen-Tsun Mathematics, Chinese Academy of Sciences,\\
 School of Mathematical Sciences, University of Science and Technology of China, Hefei 230026, Anhui, PR China\\
 E-mail: xwchen$\symbol{64}$mail.ustc.edu.cn}

 \vskip 5pt

 {\footnotesize \noindent Zengqiang Lin \\
 School of Mathematical Sciences, Huaqiao University, Quanzhou 362021, Fujian, PR China\\
 E-mail: lzq134$\symbol{64}$163.com
 }

  \vskip 5pt

 {\footnotesize \noindent Yu Zhou \\
 Yau Mathematical Sciences Center, Tsinghua University,  Beijing 100084, PR China\\
 E-mail: yuzhoumath$\symbol{64}$gmail.com
 }

\end{document}